\documentclass[a4]{amsart}
\usepackage{graphicx,amssymb}

\begin{document}

\newtheorem{theorem}{Theorem}
\newtheorem{lemma}[theorem]{Lemma}
\newtheorem{sublemma}{Sublemma}
\newtheorem{proposition}[theorem]{Proposition}
\newtheorem{corollary}[theorem]{Corollary}
\renewcommand{\thefootnote}{\fnsymbol{footnote}}

\markboth{G. T. JIN and S. PARK}
{Quadrisecant Approximation of Hexagonal Trefoil Knot}

\title{Quadrisecant Approximation of Hexagonal Trefoil Knot}

\author{GYO TAEK JIN and SEOJUNG PARK}

\address{Department of Mathematical Sciences, KAIST, Daejeon 305-701 Korea}
\email{trefoil@kaist.ac.kr, demian208@kaist.ac.kr}

\begin{abstract}
It is known that every nontrivial knot has at least two quadrisecants. Given a knot, we mark each intersection point of each of its quadrisecants. Replacing each subarc between two nearby marked points with a straight line segment joining them, we obtain a polygonal closed curve which we will call the quadrisecant approximation of the given knot. We show that for any hexagonal trefoil knot, there are only three quadrisecants, and the resulting quadrisecant approximation has the same knot type.
\end{abstract}

\keywords{knot, quadrisecant, trefoil knot, polygonal knot, quadrisecant approximation}

\subjclass[2000]{57M25}

\maketitle

\section{Preliminaries}	
A \emph{knot\/} is a locally flat simple closed curve in $\mathbb R^3$. Two knots said to be equivalent if there is an orientation preserving homeomorphism of $\mathbb R^3$ onto $\mathbb R^3$ carrying one to the other. The equivalence class of a knot under this equivalence relation is called its \emph{knot type\/}. A knot is said to be \emph{nontrivial\/}, if it does not have the knot type of a planar circle.

A \emph{quadrisecant\/} of a knot $K$ is a straight line $L$ such that $K\cap L$ has at least four components~\cite{burde-zieschang,pannwitz}.

\begin{theorem}[Pannwitz]
Every nontrivial knot has at least two quadrisecants.
\end{theorem}

A \emph{polygonal knot\/} is a knot which is the union of finitely many straight line segments. Each maximal line segment of a polygonal knot is called an \emph{edge\/} and its end points are called \emph{vertices\/}.

\begin{theorem}[Jin-Kim]
The trefoil knot\/\footnote[2]{$3_1$ in \cite{rolfsen}} can be constructed as a polygonal knot with at least six edges.
\end{theorem}

\section{Quadrisecants of a Hexagonal Trefoil Knot}

A polygonal knot is said to be in \emph{general position\/} if no three vertices are collinear and no four vertices are coplanar. It is clear that a quadrisecant of a polygonal knot in general position intersects the knot in finitely many points.

Let $K$ be a polygonal knot. A triangular disk $\Delta$ determined by a pair of adjacent edges of $K$ is said to be \emph{reducible\/} if the $\Delta$ intersects $K$ only in the two edges, and \emph{irreducible\/} otherwise. If $K$ is a polygonal knot in general position with the least number of edges in its knot type, then every triangular disk determined by a pair of adjacent edges of $K$ is irreducible.

Let $K$ denote a hexagonal knot with vertices $v_1,\ldots,v_6$, and edges $e_{i\,i+1}$ joining $v_i$ and $v_{i+1}$, for $i=1,\ldots,6$, where the subscripts are written modulo 6.
For $i=1,\ldots,6$, the triangular disk with vertices at $v_{i-1}, v_i, v_{i+1}$ will be denoted as $\Delta_i$.
In~\cite{huh-jeon}, Huh and Jeon showed that in a hexagonal trefoil knot, the edges and the triangular disks intersect in a special pattern as follows:

\begin{proposition}[Huh-Jeon]\label{prop:huh-jeon}
If $K$ is a hexagonal trefoil knot, then the triangular disks $\Delta_{1},\ldots,\Delta_{6}$ are irreducible. Furthermore, up to a cyclic relabeling of the vertices, the following are the only nonempty intersections among edges and triangular disks:
$$
\begin{aligned}
\operatorname{int}e_{23}\cap\operatorname{int}\Delta_{5},& &\operatorname{int}e_{23}\cap\operatorname{int}\Delta_{6},\\
\operatorname{int}e_{45}\cap\operatorname{int}\Delta_{1},& &\operatorname{int}e_{45}\cap\operatorname{int}\Delta_{2},\\
\operatorname{int}e_{61}\cap\operatorname{int}\Delta_{3},& &\operatorname{int}e_{61}\cap\operatorname{int}\Delta_{4}.
\end{aligned}
$$
\end{proposition}

\begin{figure}[h]
\newcommand{\edges}{
\put(72,49){$e_{12}$}
\put(-2,21){$e_{61}$}
\put(30,-2){$e_{56}$}
\put(46,75){$e_{45}$}
\put(12,67){$e_{34}$}
\put(72,11){$e_{23}$}
}
\centering
\begin{picture}(108,102)(-5,-5)
\edges
\put(25,15){${\Delta_5}$}
\includegraphics[scale=0.6]{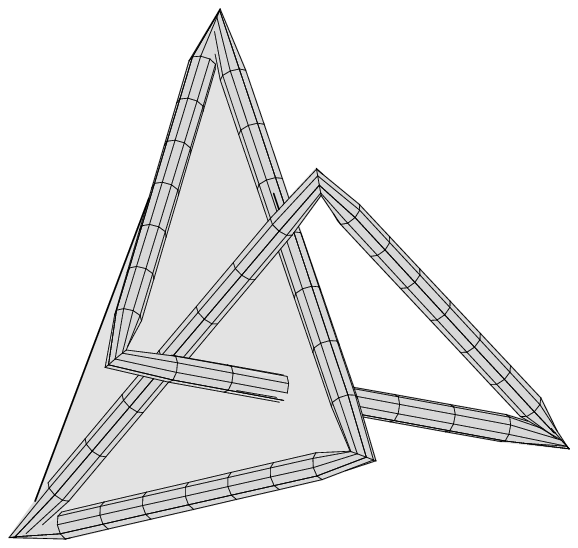}
\end{picture}
\qquad
\begin{picture}(108,102)(-5,-5)
\edges
\put(25,15){${\Delta_6}$}
\includegraphics[scale=0.6]{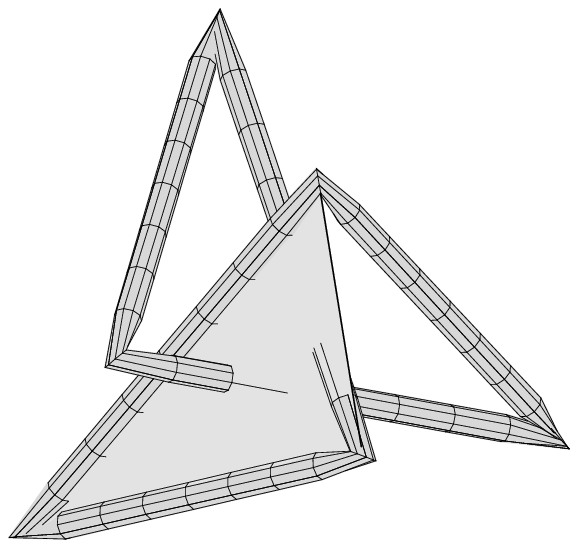}
\end{picture}
\caption{
$e_{23}\cap\Delta_{5}$ and
$e_{23}\cap\Delta_{6}$}
\end{figure}

\begin{theorem}\label{thm:only-3}
Every hexagonal trefoil knot has exactly three quadrisecants.
\end{theorem}

\begin{proof}
We first show that no three consecutive edges of $K$ are coplanar. Suppose that three consecutive edges of $K$, say $e_{12}, e_{23}$ and $e_{34}$, lie on a plane $P$. Then all vertices other than $v_5$ and $v_6$ lie on $P$. Therefore the height function of $K$ in a normal direction to $P$ has only one local maximum point and one local minimum point. Since such a knot is trivial, it contradicts that $K$ is a trefoil knot.

If $L$ is a quadrisecant of $K$, then there are four distinct edges of $K$ corresponding to four points of $K\cap L$. If any three of these edges are consecutive along $K$, then they are coplanar. Therefore, by the above argument, there are only three possible sets of four edges meeting $L$:
$$
\{e_{12},e_{23},e_{45},e_{56}\},\
\{e_{23},e_{34},e_{56},e_{61}\},\
\{e_{34},e_{45},e_{61},e_{12}\}.$$

We show that there exists exactly one quadrisecant in each of the above three cases.
We may assume that the vertices of $K$ are labeled so that the edges and the triangular disks of $K$ intersect as stated in Proposition~\ref{prop:huh-jeon}.

By cyclically relabeling the vertices of $K$, we only need to show that there exists exactly one quadrisecant meeting the edges $\{e_{12},e_{23},e_{45},e_{56}\}$.
Let $P_2$ and $P_5$ be the planes containing the triangular disks $\Delta_2$ and $\Delta_5$, respectively.
We show that $P_2\cap P_5$ is the quadrisecant we are seeking.

Notice that $e_{45}\cap P_2\ne\emptyset$ and $e_{23}\cap P_5\ne\emptyset$.
Let $p=e_{45}\cap P_2$ and $q=e_{23}\cap P_5$.
Then $P_2\cap P_5$ is the line through the two points $p$ and $q$.
Notice that the endpoints of $e_{23}$ lie on the opposite sides of $P_5$.
Let $P_5^{+}$ and $P_5^{-}$ be the open half spaces divided by $P_5$ containing $v_{2}$ and $v_{3}$, respectively. Since $v_{3}\in P_5^{-}$ and $v_{4},v_{5}\in P_5$, we see that $\Delta_{4}$ lies in $P_5^{-}$ except along $e_{45}$.
Since $e_{61}$ and $\Delta_{4}$ intersect in their interiors, we also know that $e_{61}$ lies in $P_5^{-}$ except at $v_{6}$, hence $v_{1}\in P_5^{-}$. Then, we see that $e_{12}$ intersects $P_5$ in its interior. Let $a = e_{12}\cap P_5$.
Then $a\in P_2\cap P_5$.

Notice that the endpoints of $e_{45}$ lie on the opposite sides of $P_2$. Let $P_2^{+}$ and $P_2^{-}$ be the open half spaces divided by $P_2$ containing $v_{5}$ and $v_{4}$, respectively.
Similarly as above, $\Delta_1$ lies in $P_2^-$ except along $e_{12}$. Since $e_{61}$ and $\Delta_3$ intersect in their interiors, we also know that $e_{61}$ lies in $P_2^-$ except at $v_1$, hence $v_6\in P_2^-$. Therefore we have another point $b = e_{56}\cap P_2 $ in $P_2\cap P_5$. This completes the proof.
\end{proof}

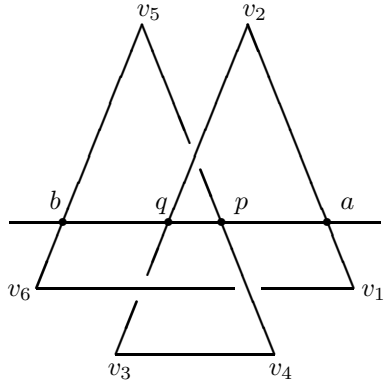
\begin{figure}[t]
\centering
\begin{picture}(150,140)(-75,-8)
\put(-70,50){\line(1,0){140}}
\put(-55,55){$b$}\put(-50,50){\circle*{3}}
\put(-15,55){$q$}\put(-10,50){\circle*{3}}
\put( 15,55){$p$}\put( 10,50){\circle*{3}}
\put( 55,55){$a$}\put( 50,50){\circle*{3}}
\thicklines
\put(-33, -8){$v_3$} \put(-30,0){\line(1,0){60}}
\put( 28, -8){$v_4$} \put(30,0){\line(-2,5){28}} \put(-20,125){\line(2,-5){18}}
\put(-22,128){$v_5$} \put(-20,125){\line(-2,-5){40}}
\put(-71, 22){$v_6$} \put(-60,25){\line(1,0){75}} \put(60,25){\line(-1,0){35}}
\put( 63, 22){$v_1$} \put(60,25){\line(-2,5){40}}
\put( 18,128){$v_2$} \put(20,125){\line(-2,-5){38}} \put(-30,0){\line(2,5){8}}
\end{picture}
\caption{$p,q,a,b\in P_2\cap P_5$}\label{fig:pqab}
\end{figure}

\section{Quadrisecant Approximation}
Let $K$ be a knot which has finitely many quadrisecants intersecting $K$ in finitely many points. The intersection points cut $K$ into finitely many subarcs. Straightening each subarc with its endpoints fixed, one obtains a polygonal closed curve which we will call the \emph{quadrisecant approximation\/} of $K$, denoted by $\hat K$.

Experiments on some knots with small crossings indicate that it might be true that $\hat K$ is actually a knot having the knot type of $K$~\cite{jin2005}. We show that the quadrisecant approximation $\hat K$ of a hexagonal trefoil knot $K$ is a trefoil knot of the same type and that the quadrisecants of $\hat K$ are those three of $K$ constructed in the proof of Theorem~\ref{thm:only-3}.

\begin{figure}[h]
\centering
\def\secant{
\put(0,0){\line(1,0){60}}
\put(9,0){\circle*{3}}\put(23,0){\circle*{3}}\put(37,0){\circle*{3}}\put(51,0){\circle*{3}}
}
\begin{picture}(60,60)(0,-50)
\secant
\qbezier(9,0)(16,20)(23,0)
\qbezier(23,0)(30,-20)(37,0)
\qbezier(37,0)(44,20)(51,0)
\qbezier(9,0)(2,-20)(9,-30)
\qbezier(9,-30)(15,-40)(30,-40)
\qbezier(51,-30)(45,-40)(30,-40)
\qbezier(51,0)(58,-20)(51,-30)
\put(20,-50){\small{simple}}

\end{picture}
\quad
\begin{picture}(60,60)(0,-50)
\secant
\qbezier(9,0)(16,20)(23,0)
\qbezier(37,0)(44,20)(51,0)
\qbezier(9,0)(2,-20)(9,-30)
\qbezier(23,0)(30,-20)(37,-30) \qbezier(37,-30)(44,-35)(51,-30)
\qbezier(37,0)(30,-20)(23,-30) \qbezier(23,-30)(16,-35)(9,-30)
\qbezier(51,0)(58,-20)(51,-30)
\put(18,-50){\small{flipped}}
\end{picture}
\quad
\begin{picture}(60,60)(0,-50)
\secant
\qbezier(9,0)(16,20)(37,0)
\qbezier(23,0)(44,20)(51,0)
\qbezier(9,0)(2,-20)(9,-30)
\qbezier(9,-30)(15,-40)(30,-40)
\qbezier(51,-30)(45,-40)(30,-40)
\qbezier(51,0)(58,-20)(51,-30)
\qbezier(37,0)(51,-10)(44,-20)
\qbezier(23,0)(9,-10)(16,-20)
\qbezier(16,-20)(30,-40)(44,-20)
\put(10,-50){\small{alternating}}
\end{picture}
\caption{Three types of quadrisecants}\label{fig:types of qsec}
\end{figure}
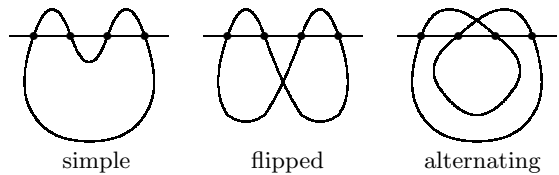

For any knot, a quadrisecant is one of the three types, \emph{simple\/}, \emph{flipped\/} or \emph{alternating\/}, according to the orders of the four intersection points along the quadrisecant and along the knot as indicated in Figure~\ref{fig:types of qsec}~\cite{BCSS}.
It is easily seen from the proof of Theorem~\ref{thm:only-3} and Figure~\ref{fig:pqab} that the following lemma holds.
\begin{lemma}\label{lem:alternating}
All quadrisecants of a hexagonal trefoil knot are alternating.
\end{lemma}

\begin{theorem} If $K$ is a hexagonal trefoil knot, then its quadrisecant approximation $\hat K$ is also a trefoil knot that has the same knot type as $K$.
\end{theorem}

\begin{figure}[h]
\centering
\includegraphics[scale=0.8]{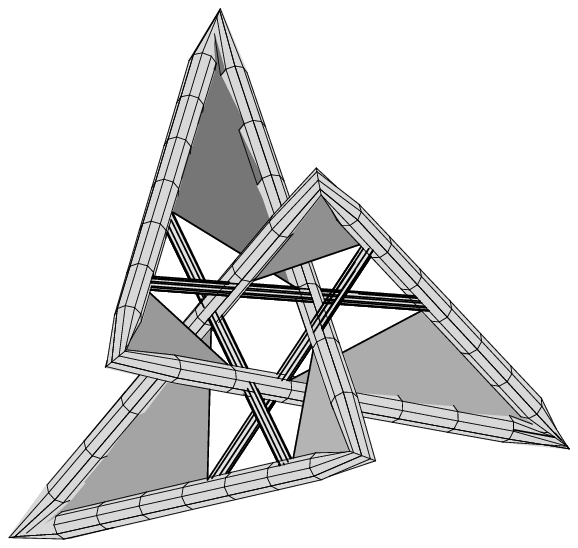}\qquad%
\includegraphics[scale=0.8]{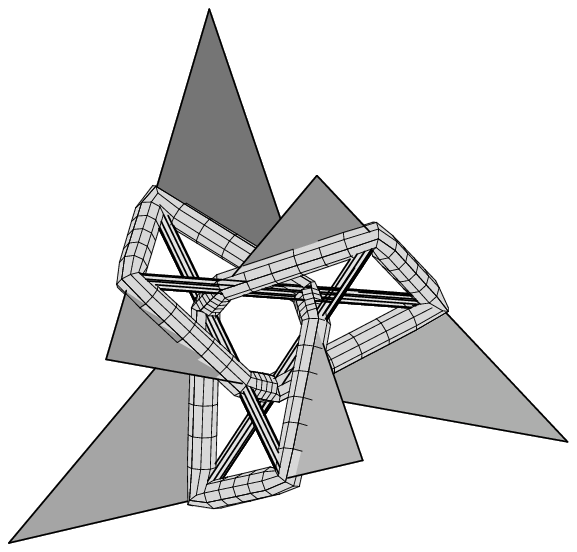}
\caption{Quadrisecants and quadrisecant appriximation of a hexagonal trefoil knot}\label{fig:corners}
\end{figure}

\begin{proof}
Let $K$ be a hexagonal trefoil knot whose vertices and edges are labeled so that it has the intersection pattern as described in Proposition~\ref{prop:huh-jeon}.
Let $L_{1},L_{2},L_{3}$ be the quadrisecants of $K$ corresponding to the set of edges
$$
\{e_{12},e_{23},e_{45},e_{56}\},\
\{e_{23},e_{34},e_{56},e_{61}\},\
\{e_{34},e_{45},e_{61},e_{12}\},$$
respectively. Let $p^k_{ij}$ denote the point $L_k\cap e_{ij}$ and let $s^k_i$ denote the subarc of $K$ which is the union of the segment of $e_{i\,i+1}$ from $v_i$ to $p^k_{i\,i+1}$ and the segment of $e_{i-1\,i}$ between $v_i$ and $p^k_{i-1\, i}$.

We first observe the quadrisecant $L_1=P_2\cap P_5$ and the edges $e_{12},e_{23},e_{45},e_{56}$. Notice that $K\cap L_1=\{p^1_{12},p^1_{23},p^1_{45},p^1_{56}\}$.
We consider $s^1_2$. If there is no $p_{ij}^{k}$'s in the interior of $s^1_2$, the quadrisecant approximation $\hat{K}$ has a self-intersection at $p_{45}^{1}$. So, we need to show that there exists one of the $p_{ij}^{k}$'s in
the interior of $s^1_2$.
Notice that $e_{34}$ lies in $P_{5}^{-}$ except at $v_{4}$. Thus, $p_{34}^{3}$, the intersection point of $L_{3}$ and $e_{34}$, lies in $P_{5}^{-}$. And note that $p_{45}^{3}$ lies on the plane $P_{5}$. By Lemma~\ref{lem:alternating}, $L_{3}$ is an alternating quadrisecant.
That is, along $L_{3}$, the order of the $p_{ij}^{3}$'s is $p_{12}^{3}p_{45}^{3}p_{61}^{3}p_{34}^{3}$. Thus, we know that $p_{12}^{3}$ lies in $P_{5}^{+}$. So, we know that $p_{12}^{3}$ lies on the intersection of $e_{12}$
and the interior of $s^1_2$.

Now, we need to show that there are no intersection points of
interior of $s^1_2$ and quadrisecants of $K$ except at $p_{12}^{3}$. Note that each of $e_{ij}$'s has exactly two
$p_{ij}^{k}$'s. So, it is not hard to see that $p_{23}^{2}$ is the only candidate for any additional intersection of $s^1_2$ and quadrisecants of $K$. We show that $p_{23}^{2}$ doesn't lie on
$s^1_2$.

By Lemma~\ref{lem:alternating}, along $L_{2}$, the order of $p_{ij}^{2}$'s is
$p_{56}^{2}p_{23}^{2}p_{61}^{2}p_{34}^{2}$.
Let $\overline{L_{2}}$ be a segment which has the end points
$p_{56}^{2}$ and $p_{34}^{2}$. Since $p_{34}^{2}$ lies on $e_{34}$
and $e_{34}$ lies in $P_{5}^{-}$, $p_{34}^{2}$ lies in $P_{5}^{-}$.
And note that $p_{56}^{2}$ lies on $P_{5}$. Thus, $p_{23}^{2}$, the
point which lies between $p_{56}^{2}$ and $p_{34}^{2}$, lies in
$P_{5}^{-}$. So, we know that $p_{23}^{2}$ does not lie on $s^1_2$.

Now, by a similar way, we can show that $p_{12}^{1}$,
$p_{34}^{3}$, $p_{34}^{2}$, $p_{56}^{2}$, $p_{56}^{1}$ are the only
points which lie on the interior of $s^3_{1}, s^2_{3}, s^3_{4}, s^1_{5},
s^2_{6}$, respectively.

Finally, let $t^k_{i}$ be the line segment joining
the end points of $s^k_{i}$. Let $\delta_{ik}$ be the triangle
which is bounded by $s^k_{i}$ and $t^k_{i}$. Since
$\operatorname{int}\Delta_{a}$ intersects only one edge of $K$ transversely and the
edge meets $t^k_{i}$, the interior of $\delta_{ik}$
does not meet $K$.

Let $\delta_{i}$ be the triangle with vertices $v_{i}$, $p_{i-1\,i}^{k}$ and $p_{i\,i+1}^{l}$ where $k$, $l$ are determined by the property:
\begin{quote} The segments $\overline{v_{i}p_{i-1\,i}^{k}}$ and $\overline{v_{i}p_{i\,i+1}^{l}}$ do not contain any quadrisecant point in their interior.
\end{quote}
Note that $\delta_{i}\subset \delta_{ik}$. Thus
$\delta_{i}$ does not meet $K$.\\
We consider the case of $\delta_{2}$. Note that $\Delta_{2}$ does
not meet $\operatorname{int}\Delta_{1}$, $\operatorname{int}\Delta_{3}$. Thus, we just
observe whether each of $\delta_{4}$, $\delta_{5}$, $\delta_{6}$ meets $\delta_{2}$ or not.
Note that $p_{45}^{3}$, $v_{4}$ and $p_{34}^{2}$ lie in $P_{2}^{-}$,
$p_{45}^{1}$ lies on $P_{2}$, $v_{5}$ and $p_{56}^{2}$ lie in
$P_{2}^{+}$. Since $p_{45}^{1}\nsubseteq\delta_{2}$, we have $(\delta_{4}\cup\delta_{5})\cap\delta_{2}=\emptyset$.

Since $\operatorname{int}\Delta_{3}$
lies on $P_{2}^{-}$ and $e_{61}$ intersects $\operatorname{int}\Delta_{3}$ transversely, we know
that $p_{61}^{2}$ and $v_{6}$ lie in $P_{2}^{-}$. And we also know
that $p_{56}^{1}$ lies on $P_{2}$ but $p_{56}^{1}\nsubseteq\delta_{2}$. So, $\delta_{6}$
does not meet $\delta_{2}$. Thus, we conclude
that $\delta_{2}$ does not meet
$\delta_{i}$.($i\neq2$).
By the similar way, we say that if $i\neq j$,
$\delta_{i}$ does not meet
$\delta_{j}$.
So, we conclude that the quadrisecant approximation
$\hat{K}$ of $K$ has the same knot type of
$K$. Thus, $\hat{K}$ is a trefoil knot.
\end{proof}

\begin{theorem}\label{thm:no more qsec}
 If $K$ is a hexagonal trefoil knot, then the quadrisecants of the quadrisecant approximation $\hat K$ are just the three quadrisecants of $K$.
\end{theorem}

The proof of Theorem~\ref{thm:no more qsec} is a combination of the lemmas and corollaries that follow. Recall that $P_i$ is the plane determined by the vertices $v_{i-1}$, $v_i$ and $v_{i+1}$. Let $O_i$ denote the edge of $\hat K$ which is contained in $e_{i\,i+1}$, and let $N_i$ denote the edge of $\hat K$ joining $O_{i-1}$ and $O_i$. $\hat K$ has twelve vertices $v_{i-1\,i}=O_{i-1}\cap N_i$ and $v_{i\,i}=N_i\cap O_i$.

Counting the number of adjacent pairs of edges among four edges meeting a quadrisecant in four distinct points, there are three types of quadrisecants.

\begin{figure}[h]
\centering
\includegraphics[scale=0.5]{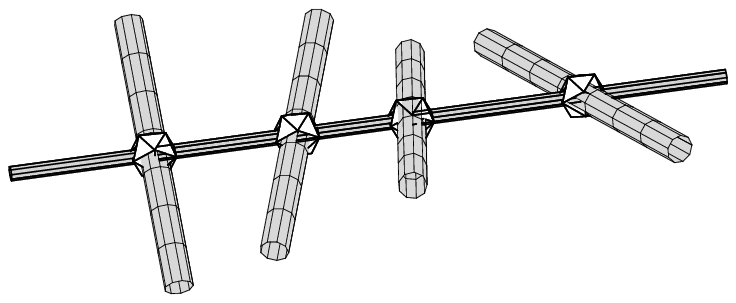}\quad
\includegraphics[scale=0.5]{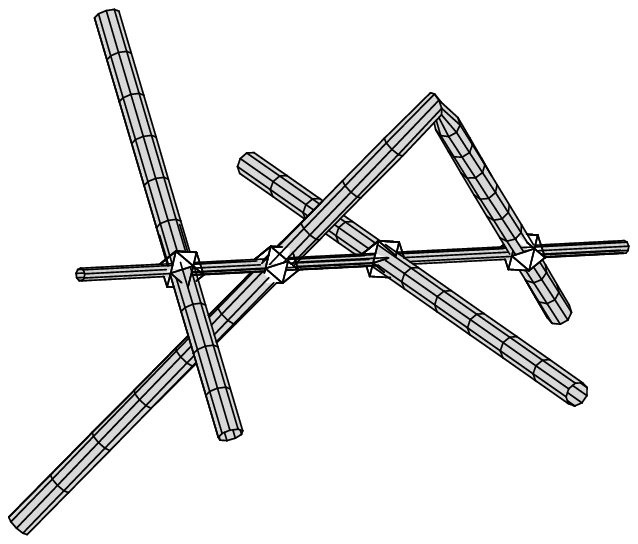}\quad
\includegraphics[scale=0.5]{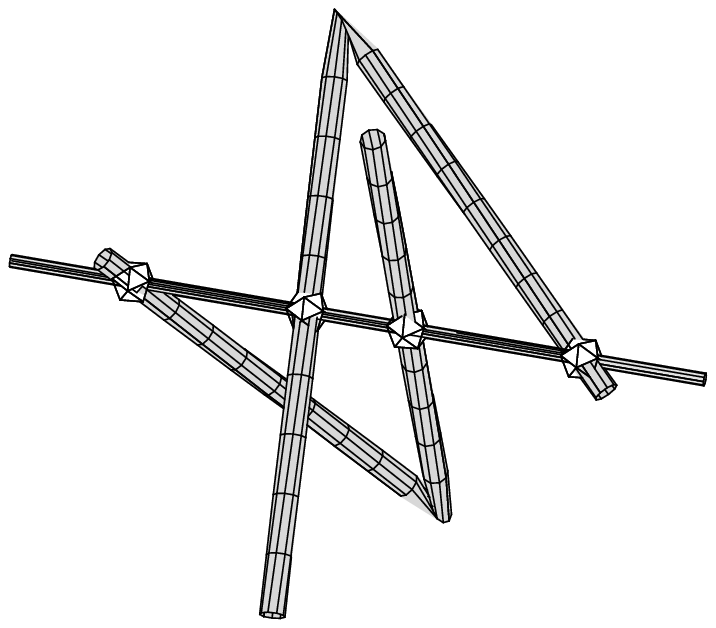}
\caption{Quadrisecants of type-0, type-1 and type-2}
\end{figure}

If a quadrisecants intersects a polygonal knot at one or more vertices, then it can have more than one types.

\begin{lemma}
If there exists a new quadrisecant of type-1 or type-2 of $\hat K$, then it lies on at least one of the $P_i$'s.
\end{lemma}

\begin{proof}
The two adjacent edges of $\hat K$ are either $O_{i-1}\cup N_i$ or $N_i\cup O_i$. As they are contained in $P_i$, the lemma holds.
\end{proof}

\begin{lemma}
For any $P_i$, there is no new quadrisecant lying on it.
\end{lemma}
\begin{proof}
First, there is no line meeting $O_{i-1}$, $N_i$ and $O_i$ except at $v_{i\,i}$ and $v_{i-1\,i}$. Next, note that there is no new quadrisecant meeting two points among the four points lying on the same quadrisecant $L_k$ of $K$.
\end{proof}

By the two previous lemmas, we have the following corollaries.

\begin{corollary}
There is no new quadrisecant of type-1 or type-2 for $\hat K$.
\end{corollary}

\begin{corollary}\label{cor:Oi-1Oi}
There is no new quadrisecant of type-0 meeting $O_{i-1}$ and $O_i$.
\end{corollary}

\begin{lemma}
There is no new quadrisecant of type-0 for $\hat K$.
\end{lemma}
\begin{proof}
By Corollary~\ref{cor:Oi-1Oi}, we know that there is no quadrisecant of type-0 meeting $O_{i-1}$ and $O_i$.
Thus, if there exists a quadrisecant of type-0 meeting $O_{i}-O_{j}-O_{k}$, then it must be $O_{1}-O_{3}-O_{5}$ or $O_{2}-O_{4}-O_{6}$.
Note that $O_{1}$ meets $N_{6}$ and $N_{1}$, $O_{3}$ meets $N_{2}$ and $N_{3}$ and $O_{5}$ meets $N_{4}$ and $N_{5}$. So, there is no quadrisecant of type-0 meeting $O_{1}-O_{3}-O_{5}$.
Similarly, there is no quadrisecant of type-0 meeting $O_{2}-O_{4}-O_{6}$.
Hence, there is no quadrisecant of type-0 meeting $O_{i}-O_{j}-O_{k}$.

We only need to check the cases meeting $N_{i}-N_{j}-N_{k}$ or $N_{i}-N_{j}-O_{k}-O_{l}$.

\medskip
\noindent{\textsc{Case 1.}} $N_{i}-N_{j}-N_{k}$ is not possible for any pairwise distinct triple $\{i,j,k\}$.

The following diagram indicates the case $\{i,j,k\}=\{1,2,3\}$ (and its order 3 cyclic relabelings $\{3,4,5\}$ and $\{5,6,1\}$). The edges $N_j$ and $N_k$ are on the same side of $P_i$. Except in the case of $N_1$ and $P_3$, the edges $N_j$, $N_k$ have one endpoint on $P_i$. Notice that $N_i$ is contained in $P_i$. One can verify that there is no proper ordering of the intersections of the edges $N_i,N_j,N_k$ along a line satisfying all the three parts of the diagram.

\centerline{
\begin{picture}(300,60)(-40,0)
\put(-50,4){$(1,2,3)$}
\put(20,0){\line(0,1){30}} \put(17,34){$N_2$}
\put(50,0){\line(0,1){30}} \put(47,34){$N_3$}
\put(110,0){\line(0,1){30}} \put(107,34){$N_1$}
\put(140,0){\line(0,1){30}} \put(137,34){$N_3$}
\put(200,0){\line(0,1){30}} \put(197,34){$N_2$}
\put(230,10){\line(0,1){20}} \put(227,34){$N_1$}
\thicklines
\put(0,4){$P_1$} \put(0,0){\line(1,0){70}}
\put(90,4){$P_2$} \put(90,0){\line(1,0){70}}
\put(180,4){$P_3$} \put(180,0){\line(1,0){70}}
\end{picture}
}

\medskip
For the other cases (and their order 3 cyclic relabelings), we only show diagrams.
\medskip
\centerline{
\begin{picture}(300,90)(-40,-40)
\put(-50,4){$(1,2,4)$}
\put(20,0){\line(0,1){30}} \put(17,34){$N_2$}
\put(50,0){\line(0,-1){30}} \put(47,-39){$N_4$}
\put(110,0){\line(0,1){30}} \put(107,34){$N_1$}
\put(140,10){\line(0,1){20}} \put(137,34){$N_4$}
\put(200,0){\line(0,1){30}} \put(197,34){$N_1$}
\put(230,0){\line(0,-1){30}} \put(227,-39){$N_2$}
\thicklines
\put(0,4){$P_1$} \put(0,0){\line(1,0){70}}
\put(90,4){$P_2$} \put(90,0){\line(1,0){70}}
\put(180,4){$P_4$} \put(180,0){\line(1,0){70}}
\end{picture}
}

\centerline{
\begin{picture}(300,90)(-40,-40)
\put(-50,4){$(1,2,5)$}
\put(20,0){\line(0,1){30}} \put(17,34){$N_2$}
\put(50,10){\line(0,1){20}} \put(47,34){$N_5$}
\put(110,0){\line(0,1){30}} \put(107,34){$N_5$}
\put(140,00){\line(0,-1){30}} \put(137,-39){$N_1$}
\put(200,0){\line(0,1){30}} \put(197,34){$N_2$}
\put(230,0){\line(0,-1){30}} \put(227,-39){$N_1$}
\thicklines
\put(0,4){$P_1$} \put(0,0){\line(1,0){70}}
\put(90,4){$P_2$} \put(90,0){\line(1,0){70}}
\put(180,4){$P_5$} \put(180,0){\line(1,0){70}}
\end{picture}
}

\centerline{
\begin{picture}(300,60)(-40,0)
\put(-50,4){$(1,2,6)$}
\put(20,0){\line(0,1){30}} \put(17,34){$N_2$}
\put(50,0){\line(0,1){30}} \put(47,34){$N_6$}
\put(110,0){\line(0,1){30}} \put(107,34){$N_1$}
\put(140,0){\line(0,1){30}} \put(137,34){$N_6$}
\put(200,0){\line(0,1){30}} \put(197,34){$N_1$}
\put(230,10){\line(0,1){20}} \put(227,34){$N_2$}
\thicklines
\put(0,4){$P_1$} \put(0,0){\line(1,0){70}}
\put(90,4){$P_2$} \put(90,0){\line(1,0){70}}
\put(180,4){$P_6$} \put(180,0){\line(1,0){70}}
\end{picture}
}

\centerline{
\begin{picture}(300,60)(-40,0)
\put(-50,4){$(1,3,5)$\footnotemark[3]}
\put(20,0){\line(0,1){30}} \put(17,34){$N_3$}
\put(50,10){\line(0,1){20}} \put(47,34){$N_5$}
\put(110,0){\line(0,1){30}} \put(107,34){$N_5$}
\put(140,10){\line(0,1){20}} \put(137,34){$N_1$}
\put(200,0){\line(0,1){30}} \put(197,34){$N_1$}
\put(230,10){\line(0,1){20}} \put(227,34){$N_3$}
\thicklines
\put(0,4){$P_1$} \put(0,0){\line(1,0){70}}
\put(90,4){$P_3$} \put(90,0){\line(1,0){70}}
\put(180,4){$P_5$} \put(180,0){\line(1,0){70}}
\end{picture}
}
\footnotetext[3]{Stable under order 3 relabeling}

\centerline{
\begin{picture}(300,90)(-40,-40)
\put(-50,4){$(1,3,6)$}
\put(20,0){\line(0,1){30}} \put(17,34){$N_3$}
\put(50,0){\line(0,1){30}} \put(47,34){$N_6$}
\put(110,0){\line(0,1){30}} \put(107,34){$N_6$}
\put(140,-10){\line(0,-1){20}} \put(137,-39){$N_1$}
\put(200,0){\line(0,1){30}} \put(197,34){$N_1$}
\put(230,0){\line(0,-1){30}} \put(227,-39){$N_3$}
\thicklines
\put(0,4){$P_1$} \put(0,0){\line(1,0){70}}
\put(90,4){$P_3$} \put(90,0){\line(1,0){70}}
\put(180,4){$P_6$} \put(180,0){\line(1,0){70}}
\end{picture}
}

\centerline{
\begin{picture}(300,90)(-40,-40)
\put(-50,4){$(1,4,6)$}
\put(20,0){\line(0,1){30}} \put(17,34){$N_4$}
\put(50,0){\line(0,-1){30}} \put(47,-39){$N_6$}
\put(110,0){\line(0,1){30}} \put(107,34){$N_1$}
\put(140,-10){\line(0,-1){20}} \put(137,-39){$N_6$}
\put(200,0){\line(0,1){30}} \put(197,34){$N_1$}
\put(230,0){\line(0,1){30}} \put(227,34){$N_4$}
\thicklines
\put(0,4){$P_1$} \put(0,0){\line(1,0){70}}
\put(90,4){$P_4$} \put(90,0){\line(1,0){70}}
\put(180,4){$P_6$} \put(180,0){\line(1,0){70}}
\end{picture}
}

\centerline{
\begin{picture}(300,60)(-40,-10)
\put(-50,4){$(2,4,6)$\footnotemark[3]}
\put(20,0){\line(0,1){30}} \put(17,34){$N_6$}
\put(50,10){\line(0,1){20}} \put(47,34){$N_4$}
\put(110,0){\line(0,1){30}} \put(107,34){$N_2$}
\put(140,10){\line(0,1){20}} \put(137,34){$N_6$}
\put(200,0){\line(0,1){30}} \put(197,34){$N_4$}
\put(230,10){\line(0,1){20}} \put(227,34){$N_2$}
\thicklines
\put(0,4){$P_2$} \put(0,0){\line(1,0){70}}
\put(90,4){$P_4$} \put(90,0){\line(1,0){70}}
\put(180,4){$P_6$} \put(180,0){\line(1,0){70}}
\end{picture}
}

\noindent{\textsc{Case 2.}} $N_{i}-N_{j}-O_{k}-O_{l}$ is not possible for any $i,j,k$ and $l$  with $i\neq j, k\neq l,l\pm1$.

In each of the following diagrams, we also consider order 3 cyclic relabelings and reverse cyclic relabelings.
One can verify that there is no proper ordering of the intersections of the edges $N_i,N_j,O_k,O_l$ along a line satisfying all the four parts of each of the diagrams below.

\centerline{
\begin{picture}(330,90)(-70,-40)
\put(-70,4){$(O_1O_3N_5N_6)$}
\put(15,0){\line(0,1){30}} \put(12,34){$N_6$}
\put(30,10){\line(0,1){20}} \put(27,34){$N_5$}
\put(45,0){\line(0,-1){30}} \put(42,-39){$O_3$}
\put(85,0){\line(0,1){30}} \put(82,34){$N_5$}
\put(100,10){\line(0,1){20}} \put(97,34){$O_1$}
\put(115,0){\line(0,-1){30}} \put(112,-39){$N_6$}
\put(155,0){\line(0,1){30}} \put(152,34){$N_6$}
\put(170,10){\line(0,1){20}} \put(167,34){$O_3$}
\put(185,0){\line(0,-1){30}} \put(182,-39){$O_1$}
\put(225,0){\line(0,1){30}} \put(222,34){$N_5$}
\put(240,10){\line(0,1){20}} \put(237,34){$O_1$}
\put(255,0){\line(0,-1){30}} \put(252,-39){$O_3$}
\thicklines
\put(0,4){$P_1$} \put(0,0){\line(1,0){50}}
\put(70,4){$P_3$} \put(70,0){\line(1,0){50}}
\put(140,4){$P_5$} \put(140,0){\line(1,0){50}}
\put(210,4){$P_6$} \put(210,0){\line(1,0){50}}
\end{picture}
}

\centerline{
\begin{picture}(330,90)(-70,-40)
\put(-70,4){$(O_1O_4N_3N_6)$}
\put(15,0){\line(0,1){30}} \put(12,34){$N_6$}
\put(30,0){\line(0,1){30}} \put(27,34){$N_3$}
\put(45,0){\line(0,1){30}} \put(42,34){$O_4$}
\put(85,0){\line(0,1){30}} \put(82,34){$N_3$}
\put(100,10){\line(0,1){20}} \put(97,34){$N_6$}
\put(115,0){\line(0,-1){30}} \put(112,-39){$O_1$}
\put(155,10){\line(0,1){20}} \put(152,34){$O_1$}
\put(170,10){\line(0,1){20}} \put(167,34){$O_4$}
\put(185,0){\line(0,-1){30}} \put(182,-39){$N_6$}
\put(225,10){\line(0,1){20}} \put(222,34){$O_4$}
\put(240,10){\line(0,1){20}} \put(237,34){$O_1$}
\put(255,0){\line(0,-1){30}} \put(252,-39){$N_3$}
\thicklines
\put(0,4){$P_1$} \put(0,0){\line(1,0){50}}
\put(70,4){$P_4$} \put(70,0){\line(1,0){50}}
\put(140,4){$P_3$} \put(140,0){\line(1,0){50}}
\put(210,4){$P_6$} \put(210,0){\line(1,0){50}}
\end{picture}
}

\centerline{
\begin{picture}(330,90)(-70,-40)
\put(-70,4){$(O_2O_4N_1N_6)$}
\put(15,0){\line(0,1){30}} \put(12,34){$N_1$}
\put(30,0){\line(0,1){30}} \put(27,34){$N_6$}
\put(45,0){\line(0,1){30}} \put(42,34){$O_4$}
\put(85,10){\line(0,1){20}} \put(82,34){$N_6$}
\put(100,10){\line(0,1){20}} \put(97,34){$O_2$}
\put(115,0){\line(0,-1){30}} \put(112,-39){$N_1$}
\put(155,0){\line(0,1){30}} \put(152,34){$N_6$}
\put(170,0){\line(0,1){30}} \put(167,34){$O_4$}
\put(185,10){\line(0,1){20}} \put(182,34){$O_2$}
\put(225,0){\line(0,1){30}} \put(222,34){$N_1$}
\put(240,0){\line(0,1){30}} \put(237,34){$O_2$}
\put(255,10){\line(0,1){20}} \put(252,34){$O_4$}
\thicklines
\put(0,4){$P_2$} \put(0,0){\line(1,0){50}}
\put(70,4){$P_4$} \put(70,0){\line(1,0){50}}
\put(140,4){$P_1$} \put(140,0){\line(1,0){50}}
\put(210,4){$P_6$} \put(210,0){\line(1,0){50}}
\end{picture}
}

This completes the proof.
\end{proof}

\section*{Acknowledgments}
This work was supported by Korea Research Foundation (KRF-2005-041-C00064).

\end{document}